\documentclass[12pt]{amsart}  
\usepackage{a4wide,graphicx,color}  
\usepackage{amsmath,amsthm,amssymb}

\let\na\nabla  
\let\eps\varepsilon

\newcommand{\R}{{\mathbb R}} 
\newcommand{\C}{{\mathbb C}}  
\newcommand{\id}{{\mathbb I}}

\newcommand{\pare}[1]{\left(#1\right)}

\renewcommand{\d}{\partial}
\newcommand{\cfun}{\mathcal C}
\newcommand{\nlin}{\mathcal N}

\DeclareMathOperator{\sign}{sign}

\newtheorem{theorem}{Theorem}   
\newtheorem{lemma}[theorem]{Lemma}   
\newtheorem{proposition}[theorem]{Proposition}   
\newtheorem{remark}[theorem]{Remark}   
\newtheorem{corollary}[theorem]{Corollary}  
\newtheorem{definition}[theorem]{Definition}

\title[Asymptotic Behavior of NLS systems with Linear Coupling]{Asymptotic Behavior of Nonlinear Schr\"odinger Systems with Linear Coupling}
\author{Paolo Antonelli}
\address{Centro di Ricerca Matematica Ennio De Giorgi, Scuola Normale Superiore,\\
Piazza dei Cavalieri, 3,\\
56100 Pisa, Italy\\ paolo.antonelli@sns.it}
\author{Rada Maria Weish\"aupl}
\address{Faculty of Mathematics, Vienna University, Oskar-Morgenstern-Platz 1\\
1090 Wien, Austria\\ rada.weishaeupl@univie.ac.at}

\date\today

\begin{document}  

 \maketitle
\begin{abstract}
A system of two coupled nonlinear Schr\"odinger equations is treated. In addition, a linear coupling which models an external driven field described by the Rabi frequency is considered. Asymptotics for large Rabi frequency are carried out. Convergence in the appropriate Strichartz space is proven. As a consequence, the global existence of the limiting system gives a criterion for the long time behavior of the original system.
\end{abstract}

{\small
{\bf Key words:} nonlinear Schr\"odinger system, Rabi frequency, asymptotics, Strichartz estimates

{\bf AMS Subject Classification}: 35Q55, 35B40.}

\section{Introduction}
In this paper we consider a model for two-component Bose-Einstein condensates irradiated by an external electromagnetic field. The dynamics are described by two coupled nonlinear Schr\"odinger equations
\begin{equation}\label{eq:nlss1}
i\d_t\psi_j=-\frac12\Delta\psi_j+V(x)\psi_j+\beta_{jj}|\psi_j|^2\psi_j+\beta_{jk}|\psi_k|^2\psi_j+\lambda\psi_k,
\end{equation}
$j, k = 1, 2$, in $(t, x)\in\R\times\R^N$, $N = 1, 2, 3$. Here $V(x)=\frac{\gamma^2}{2}|x|^2$ is the trapping potential, $\beta_{jj}$ and $\beta_{jk}$ with $\beta_{12}=\beta_{21}$  are the (scaled) intra- and inter-specific scattering lengths, respectively, $\lambda$ is the Rabi frequency related to the external electric field.
\newline
Our main goal in this paper is to investigate the asymptotic behavior of solutions to equation \eqref{eq:nlss1} when the Rabi frequency $\lambda$ becomes very large. In particular, we shall show that, after a suitable transformation of the system, the asymptotic behavior is described by two effective coupled nonlinear Schr\"odinger equations with the same inter- and intra-specific scattering lengths.
\newline
The equations we consider here arise in the modeling of a Bose-Einstein condensate formed of atoms in two different hyperfine states in the same harmonic trap \cite{Ballagh}. A binary BEC of $^{87}$Rb atoms in different spin states was produced for the first time at JILA \cite{Myatt97}. The irradiation of the condensate with an electric field induces a linear coupling in the overlap region, which causes a Josephson-type oscillation between the two species \cite{Will99}. Such condensates are very interesting in physics, since  it is possible to measure the relative phase of one component with respect to the other one 
\cite{Hall:1998fk}. Furthermore, by controlling locally the relative phase, it is also possible to produce vortices 
\cite{williams1999preparing}, \cite{matthews1999vortices}. Another interesting application of this model is the creation of a stable \emph{BEC droplet} \cite{Saito} without using highly oscillating magnetic fields through Feshbach resonance methods: indeed the external electric field is constant but it induces an effective oscillation of the scattering lengths.
\newline
There is an extensive literature on systems of nonlinear Schr\"odinger equations, and giving here a comprehensive picture of all the results in different cases would be a quite hard task.
In particular, the case without linear coupling, i.e. $\lambda=0$ in \eqref{eq:nlss1}, is considered in most of the cases. 
\cite{FanMont, LiWuLai, MaZhao} deal with the system without trapping potential and a more general class of power-type nonlinearities, giving sufficient conditions for global existence of solutions, \cite{MR2283958} analyses the existence of ground state solutions when the nonlinearities are cubic, and \cite{MR2449345} studies vortices in the same framework. In the presence of a trapping potential we mention \cite{LinWei, ChenWei, Prytula}.
On the other hand, the case with linear coupling, i.e. $\lambda\neq0$ in \eqref{eq:nlss1}, is less treated, in 
\cite{Bao} existence and uniqueness of ground states is proved, performing also asymptotics for such solutions when 
$|\lambda|\to\infty$, whereas in \cite{Juengel11} the time-dependent system with a more general class of power-type nonlinearities is considered. Here the authors give sufficient conditions for global existence or blow up, and a semi-implicit formula shows the mass transport between the two components.
\newline
Moreover, as it will be clear in the next Sections, the analysis of system \eqref{eq:nlss1} is strictly related to the study of nonlinear Schr\"odinger equations with time-oscillating nonlinearities. This link is well known in physics, see \cite{Saito}. Cazenave, Scialom \cite{Caz10} were the first to rigorously study such problems, providing a convergence result in the highly oscillating regime, see also \cite{Abdullaev} for numerical results, \cite{Fang} for the case with an energy-critical nonlinearity, 
 \cite{CarGamPan} for a system of coupled nonlinear Schr\"odinger equations. Similar problems are also studied in some related frameworks, for example in \cite{AS10} the authors carry out an asymptotic analysis for a \emph{dispersion managed} NLS, i.e. when there is a highly oscillating coefficient in front of the Laplacian, or in \cite{MR2809621} the authors consider the KdV equation with an time-dependent nonlinearity.


In Section \ref{sect:prel} we introduce the notations and state the main result of our paper. We then dedicate Section 
\ref{sect:LGWP} to the analysis of system \eqref{eq:nlss1}, establishing the cases of global existence of solutions or possible occurrence of blow-up, depending on the different choices of the coefficients $\beta_{jk}$'s.
Then in Section \ref{sect:trans} we transform \eqref{eq:nlss1} in a way which is more suitable to study the asymptotics when 
$|\lambda|$ becomes large. We then recover the limiting system, and in Section \ref{sect:as} we prove rigorous convergence. 
\section{Preliminaries and Main Result}\label{sect:prel}
In what follows $C$ will denote a generic constant greater than 1, which may possibly change from line to line.
With $\Re$ and $\Im$ we denote the real and imaginary part of a complex number, respectively. By $z^*$ we denote the complex conjugate of $z$. The scalar product between two vectors $v, v'$ will be denoted by $\langle v, v'\rangle$.
\newline
Since we are dealing with two-component Schr\"odinger systems, we shall indicate with capital letters the two-dimensional vector fields describing the wave-function of a two-component quantum system. For example we shall write 
$\Psi^t=(\psi_1, \psi_2)$, or $\Psi_0^t=(\psi_{1, 0}, \psi_{2, 0})$, $\Psi^*=(\psi_1^*, \psi_2^*)$ and so on. Consequently, we may also denote
\begin{equation*}
|\Psi|^2=|\psi_1|^2+|\psi_2|^2.
\end{equation*}
We shall denote by $L^p(\R^N), W^{1, p}(\R^N)$, the usual Lebesgue and Sobolev spaces, respectively. We shall also make use of the mixed space-time Lebesgue (or Sobolev) spaces, so that for example $L^q(I;L^r(\R^N))$ denote the space of those functions $\Psi$ having the following norm finite,
\begin{equation*}
\|\Psi\|_{L^q(I;L^r(\R^N))}:=\pare{\int_I\pare{\int_{\R^N}|\Psi(t, x)|^rdx}^{q/r}dt}^{\frac{1}{r}}.
\end{equation*}
We often shorten notation $L^q_tL^r_x(I\times\R^N)=L^q(I;L^r(\R^N))$ or $L^q_tL^r_x$ if there is no source of ambiguity.
\newline
We are interested in studying system \eqref{eq:nlss1} in the \emph{energy space}, which is defined by
\begin{equation*}
\Sigma(\R^N):=\{\Psi\in H^1(\R^N);\;|\cdot|\Psi\in L^2(\R^N)\}.
\end{equation*}
Let us consider the Hamiltonian with confining potential, $H=-\frac12\Delta+\frac{\gamma^2}{2}|x|^2$. The associated propagator is $S_0(t):=e^{-itH}$, is unitary on $L^2(\R^N)$ and from Melher's formula (see for example \cite{Car05,Oh}) we may show it satisfies a dispersive estimate for short times, namely
\begin{equation}\label{eq:disp_est}
\|S_0(t)f\|_{L^\infty(\R^N)}\lesssim|t|^{-\frac{N}{2}}\|f\|_{L^1(\R^N)},\quad|t|\leq\delta,
\end{equation}
for some small $\delta>0$ depending on $\gamma$. 
\newline
Furthermore, in what follows we also need the following commutation formulas for $H$:
\begin{equation}\label{eq:comm}
[\nabla, H]=\gamma^2x,\quad[x, H]=\nabla.
\end{equation}
Although \eqref{eq:disp_est} holds only for short times, we may infer the same Strichartz estimates for $S_0(t)$ as for the free Schr\"odinger propagator without confining potential, but only locally in time, i.e. the constants appearing in the inequalities below depend on the length of the time interval. Here we state the results we need, for further details we address the reader to \cite{Car05}, Section 2.
\begin{definition}
We say $(q, r)$ is an \emph{admissible pair} if $2\leq r\leq\frac{2N}{N-2}$ ($2\leq r\leq\infty$ for $N=1$, $2\leq r<\infty$ for $N=2$), and
\begin{equation}\label{eq:adm_pair}
\frac{1}{q}=\frac{N}{2}\pare{\frac{1}{2}-\frac{1}{r}}.
\end{equation}
\end{definition}
\begin{proposition}\label{prop:strich}
Let $(q, r), (\tilde q, \tilde r)$ be two arbitrary admissible pairs. Then, for any compact time interval, we have
\begin{equation*}
\begin{aligned}
\|S_0(t)f\|_{L^q_tL^r_x(I\times\R^N)}\leq &C(|I|, r)\|f\|_{L^2(\R^N)},\\
\|\int_0^tS_0(t-s)F(s)ds\|_{L^q_tL^r_x(I\times\R^N)}\leq &C(|I|, r, \tilde r)\|F\|_{L^{\tilde q'}_tL^{\tilde r'}_x(I\times\R^N)}.
\end{aligned}
\end{equation*}
\end{proposition}
 We may rewrite the system of equations \eqref{eq:nlss1} in the following compact way
\begin{equation}\label{eq:nlss}
\left\{\begin{aligned}
i\d_t\Psi=&-\frac12\Delta\Psi+\frac{\gamma^2}{2}|x|^2\Psi+\tilde B[\Psi]\Psi+A\Psi,\\
\Psi(0)=&\Psi_0,
\end{aligned}\right.
\end{equation}
where the nonlinearity is given by the matrix
\begin{equation}\label{eq:B_tilde}
\tilde B[\Psi]=\left(\begin{array}{cc}\beta_{11}|\psi_1|^2&\beta_{12}\psi_1\psi_2^*\\
\beta_{12}\psi_1^*\psi_2&\beta_{22}|\psi_2|^2\end{array}\right),
\end{equation}
and $A$ determines the linear coupling,
\begin{equation}\label{eq:A}
A=\left(\begin{array}{cc}&\lambda\\\lambda&\end{array}\right).
\end{equation}
The energy and the mass associated to system \eqref{eq:nlss} are the following
\begin{equation}\label{eq:nlss_en2}
E(t)=\int_{\R^N}\Big(\frac12|\nabla\Psi|^2+\frac{\gamma^2}{2}|x|^2|\Psi|^2+\frac12\Psi^*\tilde B[\Psi]\Psi
+2\lambda\Re(\psi_1^*\psi_2) \Big)(x,t)dx.
\end{equation}
\begin{equation}\label{eq:nlss_mass2}
M(t)=\int_{\R^N}|\Psi(x,t)|^2dx,
\end{equation}
and they are conserved quantities along the flow of solutions to \eqref{eq:nlss}.
It is straightforward to see that, for any $\Psi\in\Sigma(\R^N)$, the energy in \eqref{eq:nlss_en2} is finite.
\newline
As we already anticipated in the Introduction, in this paper we shall prove that, after a suitable transformation of the system, the asymptotic behavior is described by two effective coupled NLS equations which have the same inter- and intra-specific scattering lengths. Indeed, the two effective NLS equations are given by
\begin{equation}\label{eq:u_sys_main}
i\d_tu_j=-\frac12\Delta u_j+\frac{\gamma^2}{2}u_j+\chi|u_j|^2u_j+\tilde{\chi} |u_k|^2u_j,\quad j, k = 1, 2,
\end{equation}
where
\begin{equation}\label{eq:chi}
\tilde\chi:=\frac{\beta_{11}+\beta_{22}}{2} \quad \chi:=\frac{\beta_{11}+2\beta_{12}+\beta_{22}}{4}.
\end{equation}
The main Theorem we prove in this paper is
\begin{theorem}\label{thm:main1}
Let $\Psi_0\in\Sigma(\R^N)$. For any $\lambda\in\R$, let $\Psi^\lambda$ be the unique maximal solution to \eqref{eq:nlss}. Let $U^t=(u_1,u_2)$ be the solution of the system \eqref{eq:u_sys_main} in $[0, S_{max})\times\R^N$, with initial datum
\begin{equation*}
\left(\begin{array}{c}u_1(0)\\u_2(0)\end{array}\right)=\left(\begin{array}{c}\frac{1}{\sqrt2}(\psi_{1, 0}+\psi_{2, 0})\\
\frac{1}{\sqrt2}(\psi_{1, 0}-\psi_{2, 0})\end{array}\right),
\end{equation*}
where $S_{max}$ is the maximal existence time for $U$. Then for any time $0<T<S_{max}$ 
\begin{itemize}
\item the solution $\Psi^\lambda$ exists in $[0, T]$ provided $|\lambda|$ is sufficiently large.
\item for any admissible pair $(q, r)$ we have
\begin{equation}\label{eq:psi_conv}
\lim_{|\lambda|\to\infty}\pare{\|\Psi^\lambda-\tilde{U}\|_{L^q_tL^r_x}+\|\nabla(\Psi^\lambda- \tilde{U})\|_{L^q_tL^r_x}
+\||\cdot|(\Psi^\lambda-\tilde{U})\|_{L^q_tL^r_x}}=0,
\end{equation}
where the asymptotics for $\Psi$ is given by
\begin{equation*}
\tilde{U}(t)=\left(\begin{array}{c}
\frac{1}{\sqrt2}e^{-i\lambda t}u_1(t)+\frac{1}{\sqrt2}e^{i\lambda t}u_2(t)\\
\frac{1}{\sqrt2}e^{-i\lambda t}u_1(t)-\frac{1}{\sqrt2}e^{i\lambda t}u_2(t)\end{array}\right),
\end{equation*}
The $L^q_tL^r_x$-norms above are taken on the $[0, T]\times\R^N$ space-time slab. In particular, convergence holds in 
$\cfun([0, T); \Sigma(\R^N))$.
\end{itemize}
\end{theorem}
\section{Local and Global Analysis of the system} \label{sect:LGWP}
In order to prove Theorem \ref{thm:main1}, we have to transform the system into a similar one, which is more suitable to study in the limit when $|\lambda|\to\infty$. We will then show that solutions of the transformed system will converge to solutions of \eqref{eq:u_sys_main}. All those systems are quite similar and at a local level they can be studied in the same way. For this reason here we write a general local well-posedness result, valid for all systems appearing in this paper. 
The result in the Proposition below is quite standard; its proof consists in just adapting the method introduced by Kato in 
\cite{Kato87} for NLS equations, to the case of systems and for Hamiltonians with a confining potential.
As it is already well explained in \cite{Caz} (see Chapter 4 and also Remark 3.3.12), the known theory for nonlinear 
Schr\"odinger equations is easily extended to systems. In \cite{Caz}, \cite{Kato87} the presence of a confining potential is not considered, however this is only a minor modification for the local well-posedness framework, thanks to Strichartz estimates stated in Proposition \ref{prop:strich} and the commutation relations \eqref{eq:comm}.
\begin{proposition}\label{prop:LWP}
Let us consider the following Cauchy problem
\begin{equation}\label{eq:nlss2}
\left\{\begin{array}{l}
i\d_t\Psi=-\frac12\Delta\Psi+\frac{\gamma^2}{2}|x|^2\Psi+\nlin(\Psi),\\
\Psi(0)=\Psi_0,
\end{array}\right.
\end{equation}
where the unknown $\Psi$ is a complex vector field. Let $\nlin\in\cfun(\C^2;\C^2)$ be such that $\nlin(0)=0$, and 
$\nlin(\Psi)=\nlin_1(\Psi)+\nlin_2(\Psi)$, where $\nlin_1, \nlin_2\in\cfun(\C^2;\C^2)$ satisfy
\begin{itemize}
\item $|\nlin_1(\Psi)-\nlin_1(\tilde\Psi)|\leq C|\Psi-\tilde\Psi|$;
\item $|\nlin_2(\Psi)-\nlin_2(\tilde\Psi)|\leq C(|\Psi|^2+|\tilde\Psi|^2)|\Psi-\tilde\Psi|$,
\end{itemize}
 for all $\Psi, \tilde\Psi\in\C^2$.
 \newline
For any $\Psi_0\in\Sigma(\R^N)$, there exist $\delta=\delta(\|\Psi_0\|_{\Sigma})>0$ and a unique solution
$\Psi\in\cfun([0, \delta];\Sigma(\R^N))$ to \eqref{eq:nlss2}. Moreover we have
\begin{equation}\label{eq:lin_pert}
\|\Psi\|_{L^\infty([0, \delta];\Sigma(\R^N))}\leq2C\|\Psi_0\|_{\Sigma(\R^N)}.
\end{equation}
Furthermore, the solution $\Psi$ can be extended to a maximal interval $[0, T_{max})$, and the blow-up alternative holds true, namely if $T_{max}<\infty$, then
$$\lim_{t\to T_{max}}\|\nabla\Psi(t)\|_{L^2(\R^N)}=\infty;$$
Finally, for any $0<T<T_{max}$ and any admissible pair $(q, r)$, we have $\Psi, \nabla\Psi,|\cdot|\Psi\in L^q([0, T];L^r(\R^N))$.
\end{proposition}
\begin{proof}
Let us define the space
\begin{equation*}
\begin{aligned}
K:=\{\Psi\;\textrm{s.t.}\;&\Psi,\nabla\Psi,|\cdot|\Psi\in L^\infty_tL^2_x([0, \delta]\times\R^N))\cap 
L^{8/N}_tL^4_x([0, \delta]\times\R^N)),\\
&\|\Psi\|_{L^\infty([0, \delta];\Sigma(\R^N))}
+\|\left(\begin{array}{c}1\\\nabla\\|\cdot|\end{array}\right)\Psi\|_{L^{8/N}_tL^4_x([0, \delta]\times\R^N)}\leq M\},
\end{aligned}
\end{equation*}
where $M, \delta>0$ will be chosen later, endowed with the distance
\begin{equation*}
d(\Psi, \tilde\Psi):=\|\Psi-\tilde\Psi\|_{L^\infty_tL^2_x([0, \delta]\times\R^N)}
+\|\Psi-\tilde\Psi\|_{L^{8/N}_tL^4_x([0, \delta]\times\R^N)}.
\end{equation*}
By the hypotheses on $\nlin(\Psi)$ we have
\begin{equation*}
\begin{aligned}
\|\nlin_1(\Psi)-\nlin_1(\tilde\Psi)\|_{L^2(\R^N)}\leq &C\|\Psi-\tilde\Psi\|_{L^2(\R^N)}\\
\|\nlin_2(\Psi)-\nlin_2(\tilde\Psi)\|_{L^{4/3}(\R^N)}\leq &C(\|\Psi\|_{L^4(\R^N)}^2+\|\tilde\Psi\|_{L^4(\R^N)}^2)\|\Psi-\tilde\Psi\|_{L^4(\R^N)}\\
\|\left(\begin{array}{c}1\\\nabla\\|\cdot|\end{array}\right)\nlin_1(\Psi)\|_{L^2(\R^N)}\leq &C
\|\left(\begin{array}{c}1\\\nabla\\|\cdot|\end{array}\right)\Psi\|_{L^2(\R^N)}\\
\|\left(\begin{array}{c}1\\\nabla\\|\cdot|\end{array}\right)\nlin_2(\Psi)\|_{L^{4/3}(\R^N)}
\leq &C\|\Psi\|_{L^4(\R^N)}^2\|\left(\begin{array}{c}1\\\nabla\\|\cdot|\end{array}\right)\Psi\|_{L^4(\R^N)}.
\end{aligned}
\end{equation*}
Let us now consider the operator
\begin{equation*}
G[\Psi]:=S_0(t)\Psi_0-i\int_0^tS_0(t-s)\nlin(\Psi)(s)ds.
\end{equation*}
By using the commutation rules for the Hamiltonian $H$ and Strichartz estimates, we have on the space-time slab $[0, \delta]\times\R^N$
\begin{equation*}
\begin{aligned}
\|G[\Psi]\|_{L^\infty_t\Sigma_x}+&
\|\left(\begin{array}{c}1\\\nabla\\|\cdot|\end{array}\right)G[\Psi]\|_{L^{8/N}_tL^4_x}\\
&\leq C\|\Psi_0\|_{\Sigma(\R^N)}
+C\|\left(\begin{array}{c}1\\\nabla\\|\cdot|\end{array}\right)\nlin_1(\Psi)\|_{L^1_tL^2_x}
+C\|\left(\begin{array}{c}1\\\nabla\\|\cdot|\end{array}\right)\nlin_2(\Psi)\|_{L^{\frac{8}{8-N}}_tL^{4/3}_x}\\
&\leq C\|\Psi_0\|_{\Sigma(\R^N)}
+C\|\left(\begin{array}{c}1\\\nabla\\|\cdot|\end{array}\right)\Psi\|_{L^1_tL^2_x}
+C\|\Psi\|_{L^\infty_tL^4_x}^2
\|\left(\begin{array}{c}1\\\nabla\\|\cdot|\end{array}\right)\Psi\|_{L^{\frac{8}{8-N}}_tL^4_x}.
\end{aligned}
\end{equation*}
Analogously we obtain
\begin{equation*}
\begin{aligned}
\|G[\Psi]-G[\tilde\Psi]\|_{L^\infty_tL^2_x}+&\|G[\Psi]-G[\tilde\Psi]\|_{L^{8/N}_tL^4_x}
\leq C\|\nlin_1(\Psi)-\nlin_1(\tilde\Psi)\|_{L^1_tL^2_x}+C\|\nlin_2(\Psi)-\nlin_2(\tilde\Psi)\|_{L^{\frac{8}{8-N}}_tL^{4/3}_x}\\
&\leq C\|\Psi-\tilde\Psi\|_{L^1_tL^2_x}+C(\|\Psi\|_{L^\infty_tL^4_x}^2+\|\tilde\Psi\|_{L^\infty_tL^4_x}^2)
\|\Psi-\tilde\Psi\|_{L^{\frac{8}{8-N}}_tL^4_x}.
\end{aligned}
\end{equation*}
By the Sobolev embedding $H^1\hookrightarrow L^4$ and by using H\"older's inequality in time in the previous expressions, we have
\begin{equation*}
\begin{aligned}
\|G[\Psi]\|_{L^\infty_t\Sigma_x}
+\|\left(\begin{array}{c}1\\\nabla\\|\cdot|\end{array}\right)G[\Psi]\|_{L^{8/N}_tL^4_x}
\leq &C\|\Psi_0\|_{\Sigma}
+C(\delta+\delta^{\frac{8-2N}{8}}M^2)M\\
\|G[\Psi_1]-G[\Psi_2]\|_{L^{8/N}_tL^4_x}\leq &
C(\delta+\delta^{\frac{8-2N}{8}}M^2)d(\Psi, \tilde\Psi).
\end{aligned}
\end{equation*}
Now, we choose $M, \delta$ such that
\begin{equation*}
\begin{aligned}
C\|\Psi_0\|_{\Sigma(\R^N)}=&\frac{M}{2}\\
C(\delta+\delta^{\frac{8-2N}{8}}M^2)\leq&\frac12,
\end{aligned}
\end{equation*}
so that we have
\begin{equation}\label{eq:21}
\begin{aligned}
&\|G[\Psi]\|_{L^\infty_t\Sigma_x}
+\|\left(\begin{array}{c}1\\\nabla\\|\cdot|\end{array}\right)G[\Psi]\|_{L^{8/N}_tL^4_x}\leq M\\
&d(G[\Psi], G[\tilde\Psi])\leq\frac12d(\Psi, \tilde\Psi).
\end{aligned}
\end{equation}
This implies $G$ is a contraction in $K$, therefore there exists a unique $\Psi\in K$ such that
\begin{equation*}
\Psi(t)=G[\Psi](t)=S_0(t)\Psi_0-i\int_0^tS_0(t-s)\nlin(\Psi)(s)ds.
\end{equation*}
Hence $\Psi\in\cfun([0, \delta];\Sigma(\R^N))\cap L^{8/N}([0, \delta];L^4(\R^N))$ is a solution to \eqref{eq:nlss2} in $[0, \delta]$.
From \eqref{eq:21} we also see
\begin{equation*}
\|\Psi\|_{L^\infty([0, \delta];\Sigma(\R^N))}\leq M=2C\|\Psi_0\|_{\Sigma(\R^N)},
\end{equation*}
which proves \eqref{eq:lin_pert}. Analogously, by Strichartz estimates we also have
\begin{equation*}
\|\left(\begin{array}{c}1\\\nabla\\|\cdot|\end{array}\right)\Psi\|_{L^q_tL^r_x([0, \delta]\times\R^N)}\leq 
C\|\Psi_0\|_{\Sigma(\R^N)}+C\delta^{\frac{8-2N}{N}}(1+M^2)M\leq M.
\end{equation*}
Furthermore, from the proof in the fixed point argument above we also infer that we may extend the solution as long as the 
$L^2-$norm of the gradient of the solution remains bounded, hence the blow-up alternative holds true. This implies we can extend the solution to a maximal interval $[0, T_{max})$, and moreover for any $0<T<T_{max}$, $(q, r)$ admissible pair, we have
\begin{equation*}
\|\left(\begin{array}{c}1\\\nabla\\|\cdot|\end{array}\right)\Psi\|_{L^q_tL^r_x([0, T]\times\R^N)}<\infty.
\end{equation*}
\end{proof}
\begin{corollary}\label{cor:lwp1}
The system \eqref{eq:nlss} is locally well-posed in $\Sigma(\R^N)$. 
\end{corollary}
A natural question which arises at this point is to see whether the solution in Corollary above is global or there is possible occurrence of blow-up in finite time.
\newline
For the case of a single cubic nonlinear Schr\"odinger equation the picture is complete.
\begin{itemize}
\item If the nonlinearity is defocusing (i.e. its coefficient is positive) then, by using the conservation of energy and the consequent uniform a priori bound on the $H^1-$norm of the solution, global well-posedness holds true.
\item For $N=3$, if the nonlinearity is focusing (i.e. its coefficient is negative), then there exist initial data for which the 
$L^2-$norm of the gradient of the solution blows up in finite time \cite{G}.
\item In the case $N=2$ and focusing nonlinearity, we have
\begin{itemize}
\item if $\|\psi_0\|_{L^2}<\|Q\|_{L^2}$, where $Q$ is the unique positive radial solution to $\Delta Q+Q^3-Q=0$, then the solution exists globally \cite{Weinstein};
\item for initial data $\|\psi_0\|_{L^2}\ge\|Q\|_{L^2}$ there is possible occurrence of blow-up.
\end{itemize}
\end{itemize}
In the case of system \eqref{eq:nlss1} the picture is more complex since each of the nonlinearity can induce blow-up. Here we state some known results.
\begin{theorem}[\cite{Juengel11}]\label{thm:GWP1}
Let $N\le 3$ and set $\beta=\max\{2(-\beta_{11})^+,2(-\beta_{22})^+,(-\beta_{12})^+\}$. 
Then there exists a global-in-time solution to 
in the following cases:
\begin{enumerate}
\item  $\beta_{11},\beta_{22},\beta_{12}\ge 0$ or $\beta_{12}^2<  \beta_{11}\beta_{22}$ with $\beta_{11}\geq0$, $\beta_{12}<0$ 
 \item $N=1$
 \item $N=2$ and $M(0) < 2/(C_{2}\beta)$, if $\min\{\beta_{11},\beta_{22},\beta_{12}\}<0$
  \item $N=3$, $ \|\na\Psi(0)\|_2^2 \le 2(E(0)+|\lambda|M(0))$, and $M(0)(E(0)+|\lambda|M(0)) < \frac{8}{27C_{3}^2\beta^2}$, if $\min\{\beta_{11},\beta_{22},\beta_{12}\}<0$
\end{enumerate}
\end{theorem}
\begin{proof}
We show global existence for $\beta_{12}^2<  \beta_{11}\beta_{22}$ with $\beta_{11}>0$. In this case indeed the matrix
\begin{equation*}
B=\left(\begin{array}{cc}\beta_{11}&\beta_{12}\\\beta_{12}&\beta_{22}\end{array}\right)
\end{equation*}
is positive definite, it follows that 
 \begin{equation*}
 \int_{\R^N}\big(\frac{\beta_{11}}{2}|\psi_2|^4+\beta_{12}|\psi_1|^2|\psi_2|^2+\frac{\beta_{22}}{2}|\psi_2|^2\big)(x,t)dx \geq 0,
 \end{equation*}
thus using the energy conservation we get uniform bounds on the $\Sigma$-norm of the solution $\Psi$.
 \begin{equation*}
 \frac{1}{2}\|\nabla\Psi(t)\|_{L^2}^2+\frac{\gamma^2}{2}\||\cdot|\Psi\|_{L^2}^2
 \leq E(t)-2\lambda\int_{\R^N}\Re(\psi_1^*\psi_2)(x,t)dx
 \leq E(t)+|\lambda| M(t)=E(0)+|\lambda| M(0).
 \end{equation*}
 For all other cases see \cite{Juengel11}.
\end{proof}
In the next Theorem we show the cases in which there is possible occurrence of blow-up in finite time. We describe two situations, the first one in which the nonlinearity is negative definite. In this case we apply a method introduced by Carles \cite{Car} for focusing NLS equations with a confining potential, by using a modified energy functional (for more discussions about this modified energy and its interpretation see \cite{Car}). In the second case we assume that at least one of the coefficients $\beta_{ij}$'s is focusing and that the energy is negative enough so that conditions 
(ii) or (iii) below are fulfilled. In this case we apply the method by Glassey \cite{G} using virial identities.
\begin{remark}
Note that the following sufficient condition for blow-up is valid also in the supercritical case $N=3$, whereas in \cite{Juengel11} the authors consider only the mass-critical case $N=2$.
\end{remark}
\begin{theorem}\label{thm:blow-up}
Let $\Psi\in\cfun([0, T_{max});\Sigma(\R^N))$ be the solution to \eqref{eq:nlss} as in Corollary \ref{cor:lwp1},  and let us define the virial potential 
\begin{equation*}
I(t)=\int_{\R^N}|x|^2|\Psi(x,t)|^2dx.
\end{equation*}
Let us assume $N\geq2$ and one of the following conditions is satisfied
\begin{itemize}
\item[(i)] the nonlinearity is \emph{negative definite}, i.e. we either have $\beta_{12}^2-\beta_{11}\beta_{22}<0$ with 
$\beta_{11}<0$ and  $\beta_{12}\geq 0$ or $\beta_{11}, \beta_{12}, \beta_{22}<0$, and we also assume that
\begin{equation*}
E(0)+\frac{|\lambda|}{2}M(0)<\frac{\gamma^2}{2}I(0);
\end{equation*}
\item[(ii)] $\min\{\beta_{11},\beta_{22},\beta_{12}\}<0$
$$\frac{2N}{N+2}\pare{E(0)+|\lambda|M(0)}<\frac{\gamma^2}{2}I(0);$$
\item[(iii)] $\min\{\beta_{11},\beta_{22},\beta_{12}\}<0$; 
$I'(0)<0$ and $$\frac{2N}{N+2}\pare{E(0)+|\lambda|M(0)}<-\frac{\gamma}{\sqrt{2+N}}I'(0).$$
\end{itemize}
Then the solution blows-up at a finite time, i.e. $\exists\;0<T^*<\infty$, such that
\begin{equation*}
\lim_{t\to T^*}\|\nabla\Psi(t)\|_{L^2}=\infty.
\end{equation*}
\end{theorem}
\begin{proof}
We first consider case (i). Analogously to \cite{Car} we introduce the following functional
\begin{equation*}
\begin{aligned}
E_1(t)=&\cos^2(\gamma t)\int_{\R^N}\big(\frac12|\nabla\Psi|^2+\frac12\Psi^*\tilde B[\Psi]\Psi+2\lambda\Re(\psi_1^*\psi_2)\big)(t, x)dx\\
&+\sin^2(\gamma t)\int_{\R^N}\frac{\gamma^2}{2}|x|^2|\Psi(t, x)|^2dx
+\frac{\gamma}{2}\sin(2\gamma t)\int_{\R^N} x\cdot J(t, x)dx\\
&+|\lambda|\frac{\cos(2\gamma t)}{2}\int_{\R^N}|\Psi(t, x)|^2dx,
\end{aligned}
\end{equation*}
where $J(t, x)$ is the current density defined by
\begin{equation*}
J(x,t)=\Im(\psi_1^*\nabla\psi_1+\psi_2^*\nabla\psi_2)(t, x).
\end{equation*}
Let us first notice that the condition $E(0)+\frac{|\lambda|}{2}M(0)<\frac{\gamma^2}{2}I(0)$ implies $E_1(0)<0$. By computing the time derivative of $E_1$ we obtain
\begin{equation*}
\frac{d}{dt}E_1(t)=\gamma\sin(2\gamma t)\frac{N-2}{4}\int_{\R^N}\big(\Psi\tilde B[\Psi]\Psi\Big) (t, x)dx
+|\lambda|\gamma\sin(2\gamma t)\int_{\R^N} \big(2\sign(\lambda)\Re(\psi_1^*\psi_2)-|\Psi|^2\big)(t, x)dx.
\end{equation*}
By the assumptions on the coefficients $\beta_{ij}$'s we have $\Psi\tilde B[\Psi]\Psi<0$ for any $\Psi\neq0$, we thus infer
\begin{equation}\label{eq:decr_en}
\frac{d}{dt}E_1(t)\leq0,\quad\forall\;t\in\left[0, \frac{\pi}{2\gamma}\right].
\end{equation}
On the other hand, we see that if $\Psi$ exists in $[0, \frac{\pi}{2\gamma}]$, then 
\begin{equation*}
E_1\pare{\frac{\pi}{2\gamma}}=\frac{\gamma^2}{2}\int|x|^2|\Psi(x,t)|^2dx\geq0,
\end{equation*}
contradicting $E_1(0)<0$ and \eqref{eq:decr_en}. Thus, there exists $T^*\leq\frac{\pi}{2\gamma}$ such that
\begin{equation*}
\lim_{t\to T^*}\|\nabla\Psi(t)\|_{L^2}=\infty.
\end{equation*}
For cases (ii) and (iii), we use the virial identities. By similar calculations as above we have
\begin{equation*}
\begin{aligned}
I'(t)=&2\int_{\R^N} x\cdot J(x,t)dx\\
I''(t)=&\int_{\R^N}\big(2|\nabla\Psi|^2+N\Psi^*\tilde B[\Psi]\Psi-2\gamma^2|x|^2|\Psi|^2\big)(t, x)dx.
\end{aligned}
\end{equation*}
We write
\begin{equation*}
I''(t)=2NE(t)+(2-N)\int_{\R^N}|\nabla\Psi(x,t)|^2dx-(2+N)\gamma^2I(t)
-4N\lambda\int_{\R^N}\Re(\psi_1^*\psi_2)(t, x)dx.
\end{equation*}
From the conservation of energy and mass we then obtain
\begin{equation*}
I''(t)=-(2+N)\gamma^2I(t)+2N(E(0)+|\lambda|M(0))+R(t),
\end{equation*}
where 
\begin{equation*}
R(t)=(2-N)\int_{\R^N}|\nabla\Psi(t, x)|^2dx-4N\lambda\int_{\R^N}\Re(\psi_1^*\psi_2)(t, x)dx-2N|\lambda|M(t)\leq0.
\end{equation*}
The solution of the above differential equation is given by
\begin{equation*}
\begin{aligned}
I(t)=&\cos(\sqrt{2+N}\gamma t)I(0)+\frac{1}{\sqrt{2+N}\gamma}\sin(\sqrt{2+N}\gamma t)I'(0)\\
&+\frac{2N}{(2+N)\gamma^2}(E(0)+|\lambda|M(0))(1-\cos(\sqrt{2+N}\gamma t))\\
&+\frac{1}{\sqrt{2+N}\gamma}\int_0^t\sin(\sqrt{2+N}\gamma (t-s))R(s)ds\\
\leq&\cos(\sqrt{2+N}\gamma t)I(0)+\frac{1}{\sqrt{2+N}\gamma}\sin(\sqrt{2+N}\gamma t)I'(0)\\
&+\frac{2N}{(2+N)\gamma^2}(E(0)+|\lambda|M(0))(1-\cos(\sqrt{2+N}\gamma t)).
\end{aligned}
\end{equation*}
With the inequality valid at least for $ t\in \left[0,\frac{\pi}{\gamma\sqrt{2+N}}\right]$. We then notice that, if (ii) or (iii) are satisfied, then $I(t)$ will eventually become negative, giving a contradiction.
The unboundedness of the gradient then follows from the conservation of mass and the uncertainty inequality
\begin{equation*}
\|\Psi(t)\|_{L^2}^2\leq\frac{2}{N}\||\cdot|\Psi(t)\|_{L^2}\|\nabla\Psi(t)\|_{L^2}
\lesssim\sqrt{I(t)}\|\nabla\Psi(t)\|_{L^2}.
\end{equation*}
\end{proof}

By taking $\lambda=0$ we can apply Theorem  \ref{thm:GWP1}  also to system \eqref{eq:u_sys_main}.  If we 
consider formulas \eqref{eq:chi} for the coefficients $\chi$, $\tilde{\chi}$, we then obtain the following global well-posedness result for  \eqref{eq:u_sys_main}.
\begin{proposition}\label{prop:u_lwp}
Let $U_0\in\Sigma(\R^N)$, then there exist  a unique, maximal solution $U\in\cfun([0, S_{max}); \Sigma(\R^N))$ for the system \eqref{eq:u_sys_main}. The usual blow-up alternative holds true and for any time $0<T<S_{max}$, admissible pair $(q, r)$, we have
\begin{equation*}
U, \nabla U, |\cdot|U\in L^q([0, T]; L^r(\R^N)).
\end{equation*}
Moreover, we have global existence, i.e. $S_{max}=\infty$, in the following cases:
\begin{itemize}
\item
  	$\beta_{11}+\beta_{22}\geq 0$ and  $ \beta_{11}+ 2\beta_{12}+\beta_{22}\geq 0$;
\item    $\beta_{11}+\beta_{22}< 0$ and $\frac{3}{2} \left| \beta_{11}+\beta_{22}\right|<\beta_{12}$.
  \end{itemize}
\end{proposition}

\section{The transformed system}\label{sect:trans}
By looking at system \eqref{eq:nlss} we may infer it is not possible to study the limit $|\lambda|\to\infty$ directly there. This is already clear  by looking at the linear part of \eqref{eq:nlss}, namely
\begin{equation}\label{eq:lin_sys}
i\d_t\Psi=-\frac12\Delta\Psi+\frac{\gamma^2}{2}|x|^2\Psi+A\Psi.
\end{equation}
While for $\lambda=0$ (i.e. $A=0$), the two equations in \eqref{eq:lin_sys} decouple and they evolve independently through the Hamiltonian $H=-\frac12\Delta+\frac{\gamma^2}{2}|x|^2$, $\Psi(t)=S_0(t)\Psi_0$, in the case $\lambda\neq0$ the two equations are coupled, and we have 
\begin{equation*}
\Psi(t)=S_\lambda(t)\Psi_0,
\end{equation*}
where 
$S_\lambda(t)=S_0(t)\Omega_\lambda(t)=e^{-itH}\Omega_\lambda(t)$, and 
\begin{equation*}
\Omega_\lambda(t):=e^{-itA}
=\left(\begin{array}{cc}\cos(\lambda t)&-i\sin(\lambda t)\\-i\sin(\lambda t)&\cos(\lambda t)\end{array}\right).
\end{equation*}
Hence, from the formula above for $S_\lambda(t)$, we see that in order to study the asymptotic behavior of solutions to 
\eqref{eq:nlss}, we first have to transform it by eliminating the oscillations in the one-parameter group family $S_\lambda(t)$.
We have
\begin{equation}\label{eq:91}
\Omega_\lambda(t)=V\left(\begin{array}{cc}e^{-i\lambda t}&\\&e^{i\lambda t}\end{array}\right)V, 
\end{equation}
where $V=\frac{1}{\sqrt2}\left(\begin{array}{cc}1&1\\1&-1\end{array}\right)$. 
\begin{lemma}\label{lemma:transf}
Let $\Psi(t)$ be a solution to \eqref{eq:nlss}. Then
\begin{equation}\label{eq:phi}
\Phi(t)=V\Omega_\lambda(-t)\Psi(t).
\end{equation}
is a solution of the following NLS system
\begin{equation}\label{eq:phi_sys1}
\left\{\begin{aligned}
i\d_t\Phi=&-\frac12\Delta\Phi+\frac{\gamma^2}{2}|x|^2\Phi+\hat B[\Phi]\Phi\\
\Phi(0)=&\Phi_0,
\end{aligned}\right.
\end{equation}
where $\Phi_0=V\Psi_0$, $\hat B[\Phi]:=\hat B^\infty[\Phi]+R^\lambda[\Phi]$, with 
\begin{equation}\label{eq:b_infty}
\hat B^\infty[\Phi]=
\frac14\left(\begin{array}{cc}(\beta_{11}+2\beta_{12}+\beta_{22})|\phi|^2&2(\beta_{11}+\beta_{22})\phi_1\phi_2^*\\
2(\beta_{11}+\beta_{22})\phi_1^*\phi_2&(\beta_{11}+2\beta_{12}+\beta_{22})|\phi_2|^2
\end{array}\right),
\end{equation}
 and
\begin{equation}\label{eq:R_lambda}
\begin{aligned}
R^\lambda[\Phi]:=&
\frac14(\beta_{11}-\beta_{22})
\left(\begin{array}{cc}
e^{-2i\lambda t}\phi_1\phi_2^*+e^{2i\lambda t}\phi_1^*\phi_2&e^{2i\lambda t}(|\phi_1|^2+|\phi_2|^2)\\
e^{-2i\lambda t}(|\phi_1|^2+|\phi_2|^2)&e^{-2i\lambda t}\phi_1\phi_2^*+e^{2i\lambda t}\phi_1^*\phi_2
\end{array}\right)\\
&+\frac14(\beta_{11}-2\beta_{12}+\beta_{22})
\left(\begin{array}{cc}&e^{4i\lambda t}\phi_1^*\phi_2\\e^{-4i\lambda t}\phi_1\phi_2^*\end{array}\right).
\end{aligned}
\end{equation}
Conversely, if $\Phi$ is a solution to \eqref{eq:phi_sys1}, then
\begin{equation*}
\Psi(t)=\Omega_\lambda(t)V\Phi(t)
\end{equation*}
is a solution to \eqref{eq:nlss}, with $\Psi_0=V\Phi_0$.
\end{lemma}
\begin{proof}
Let us start by writing system \eqref{eq:nlss} in its equivalent integral formulation
\begin{equation}\label{eq:duham}
\Psi(t)=S_\lambda(t)\Psi_0-i\int_0^tS_\lambda(t-s)\tilde B[\Psi]\Psi(s)ds.
\end{equation}
By using formula \eqref{eq:phi} we get
\begin{equation*}
\Phi(t)=S_0(t)V\Psi_0-i\int_0^tS_0(t-s)V\Omega_\lambda(-s)\tilde B[\Psi]\Omega_\lambda(s)V\Phi(s)\,ds.
\end{equation*}
We want to compute the nonlinear potential matrix 
$V^t\Omega_\lambda(-s)\tilde B[\Psi]\Omega_\lambda(s)V=:\hat B[\Phi]$.
After some tedious but straightforward calculations, and by rearranging a bit the terms we write
\begin{equation*}
\begin{aligned}
\hat B[\Phi]=&
\frac14\left[(\beta_{11}+2\beta_{12}+\beta_{22})\left(\begin{array}{cc}|\phi_1|^2&\phi_1\phi_2^*\\
\\\phi_1^*\phi_2&|\phi_2|^2\end{array}\right)
+(\beta_{11}-2\beta_{12}+\beta_{22})
\left(\begin{array}{cc}|\phi_2|^2&\\&|\phi_1|^2\end{array}\right)
\right]\\
&+\frac14(\beta_{11}-\beta_{22})\left[
\id\pare{e^{-2i\lambda t}\phi_1\phi_2^*+e^{2i\lambda t}\phi_1^*\phi_2}
+\left(\begin{array}{cc}&e^{2i\lambda t}\\e^{-2i\lambda t}&\end{array}\right)
(|\phi_1|^2+|\phi_2|^2)
\right]\\
&+\frac14(\beta_{11}-2\beta_{12}+\beta_{22})
\left(\begin{array}{cc}&e^{4i\lambda t}\phi_1^*\phi_2\\e^{-4i\lambda t}\phi_1\phi_2^*\end{array}\right)\\
=&:\hat B^\infty[\Phi]+R^\lambda[\Phi],
\end{aligned}
\end{equation*}
where $\hat B^\infty[\Phi]$ is the homogeneous part and $R^\lambda[\Phi]$ the one with time-dependent coefficients.
Let us notice that the term 
$\hat B^\infty[\Phi]\Phi$ can also be written in the following compact form
\begin{equation*}
\frac14\left(\begin{array}{cc}(\beta_{11}+2\beta_{12}+\beta_{22})|\phi|^2&2(\beta_{11}+\beta_{22})\phi_1\phi_2^*\\
2(\beta_{11}+\beta_{22})\phi_1^*\phi_2&(\beta_{11}+2\beta_{12}+\beta_{22})|\phi_2|^2
\end{array}\right)\Phi,
\end{equation*}
thus we may rename $\hat B^\infty[\Phi]$ to be the matrix in the expression above.
Consequently, the equation for the transformed variable $\Phi$ becomes
\begin{equation}\label{eq:phi_sys}
i\d_t\Phi=-\frac12\Delta\Phi+\frac{\gamma^2}{2}|x|^2\Phi+\hat B[\Phi]\Phi,
\end{equation}
with initial datum $\Phi(0)=V\Psi_0$.
\end{proof}
It is straightforward to check that the nonlinear terms $\hat B^\infty, R^\lambda$ satisfy
\begin{align}
|\hat B^\infty[F_1]F_1-\hat B^\infty[F_2]F_2|\lesssim(|F_1|^2+|F_2|^2)|F_1-F_2|,\label{eq:11}\\
|\nabla\left(\hat B^\infty[F_1]F_1-\hat B^\infty[F_2]F_2\right)|\lesssim(|F_1|^2+|F_2|^2)|\nabla(F_1-F_2)|\label{eq:12}\\
|\nabla^k(R^\lambda[F_1]F_1)|\lesssim|F_1|^2|\nabla^kF_1|,\quad k=0, 1\label{eq:13}
\end{align}
Formally, in the limit as $|\lambda|$ goes to infinity, we expect the non-homogeneous part $R^\lambda$ to cancel, because of the oscillating coefficients in front of the nonlinearities which average out to zero. We thus obtain the system of coupled nonlinear Schr\"odinger equations  \eqref{eq:u_sys_main}, which we rewrite in a compact form:
\begin{equation}\label{eq:u_sys}
i\d_tU=-\frac12\Delta U+\frac{\gamma^2}{2}|x|^2U+\hat B^\infty[U]U,
\end{equation}
with energy
\begin{equation}\label{eq:u_en}
\begin{aligned}
\hat E(t)=&\int_{\R^N}\Big(\frac12|\nabla U|^2+\frac{\gamma^2}{2}|x|^2|U|^2+\frac12U^*\hat B^\infty[U]U\Big)(x,t)dx\\
=&\int_{\R^N}\Big(\frac12|\nabla U|^2+\frac{\gamma^2}{2}|x|^2|U|^2
+\frac{\chi}{2}(|u_1|^4+|u_2|^4)
+\tilde\chi|u_1|^2|u_2|^2\Big)(x,t)dx,
\end{aligned}
\end{equation}
where $\chi, \tilde\chi$ are defined in \eqref{eq:chi}. Moreover, since there is no linear coupling in \eqref{eq:u_sys} between the two equations, then the two total component masses, namely
\begin{equation*}
\hat M_j(t)=\int_{\R^N}| u_j(x,t)|^2dx, \mbox{ with } j=1,2,
\end{equation*}
are conserved along the flow of solutions to \eqref{eq:u_sys}.
\begin{remark}
Let us consider formula \eqref{eq:duham} again, and suppose we want to eliminate the oscillations by simply considering
\begin{equation*}
\tilde\Phi(t):=\Omega_\lambda(-t)\Psi(t).
\end{equation*}
Then we obtain the following system
\begin{equation*}
\tilde\Phi(t)=S_0(t)\tilde\Phi_0-i\int_0^tS_0(t-s)B_2[\tilde\Phi]\tilde\Phi(s)ds,
\end{equation*}
but now the nonlinear potential matrix $B_2[\tilde\Phi]$ is given by the more complicated expression
\begin{equation*}
\begin{aligned}
B_2[\tilde\Phi]=&\frac14(\beta_{11}+\beta_{22})|\tilde\Phi|^2
+\frac{\beta_{12}}{2}\left(\begin{array}{cc}0&1\\1&0\end{array}\right)(|\tilde\phi_1|^2-|\tilde\phi_2|^2)\\
&+\frac18(\beta_{11}+2\beta_{12}+\beta_{22})\left[
\left(\begin{array}{cc}1&-1\\1&-1\end{array}\right)\tilde\phi_1\tilde\phi_2^*+
\left(\begin{array}{cc}1&1\\-1&-1\end{array}\right)\phi_1^*\phi_2\right]\\
&+e^{-2i\lambda t}\left\{
\frac18(\beta_{11}-\beta_{22})\left(\begin{array}{cc}1&1\\-1&-1\end{array}\right)(|\tilde\phi_1|^2+|\tilde\phi_2|^2)+
\frac14(\beta_{11}-\beta_{22})\id\tilde\phi_1\tilde\phi_2^*
\right\}\\
&+e^{2i\lambda t}\left\{
\frac18(\beta_{11}-\beta_{22})\left(\begin{array}{cc}1&-1\\1&-1\end{array}\right)(|\tilde\phi_1|^2+|\tilde\phi_2|^2)+
\frac14(\beta_{11}-\beta_{22})\id\tilde\phi_1^*\tilde\phi_2
\right\}\\
&+\frac18(\beta_{11}-2\beta_{12}+\beta_{22})\left\{
e^{-4i\lambda t}\left(\begin{array}{cc}1&1\\-1&-1\end{array}\right)\tilde\phi_1\tilde\phi_2^*
+e^{4i\lambda t}\left(\begin{array}{cc}1&-1\\1&-1\end{array}\right)\tilde\phi_1^*\tilde\phi_2
\right\}.
\end{aligned}
\end{equation*}
Furthermore, if we only consider the autonomous part, we see that even in the asymptotic limit we obtain an expression for the nonlinear potential matrix which is rather complicated.
\newline
On the other hand, if all the inter- and intra-species coefficients equal, i.e. $\beta_{11}=\beta_{12}=\beta_{22}\equiv\beta$, then the expression (and the analysis) simplifies considerably. Numerical studies using this transformation in the case of equal coefficients are performed in \cite{Decon, Nista}. Nevertheless our aim in this paper is to consider a general choice for the coefficients 
$\beta_{ij}$'s, that is why we choose to transform the system in terms of $\Phi$.
\end{remark}
\section{Asymptotics for $\lambda \rightarrow \infty$}\label{sect:as}
In this Section we prove the rigorous convergence of solutions to system \eqref{eq:phi_sys}, towards solutions to system 
\eqref{eq:u_sys}, when $|\lambda|\to\infty$. Here we follow the same strategy as in Cazenave, Scialom \cite{Caz10}.
The main result of this Section is Theorem \ref{thm:main}:
\begin{theorem}\label{thm:main}
Let $\Phi_0\in\Sigma(\R^N)$. For any $\lambda\in\R$, we denote by $\Phi^\lambda$ the unique maximal solution to 
\eqref{eq:phi_sys}. Let $U$ be the solution to \eqref{eq:u_sys}, with initial data $U(0)=\Phi_0$, in $[0, S_{max})$, where 
$0<S_{max}\leq\infty$, as in Proposition \ref{prop:u_lwp}.
\begin{itemize}
\item For any $0<T<S_{max}$, the solution $\Phi^\lambda$ exists in $[0, T]$ provided $|\lambda|$ is sufficiently large.
\item For any $0<T<S_{max}$, $(q, r)$ admissible pair, we have
\begin{equation}\label{eq:conv}
\|\Phi^\lambda-U\|_{L^q_tL^r_x}+\|\nabla(\Phi^\lambda-U)\|_{L^q_tL^r_x}+\||\cdot|(\Phi^\lambda-U)\|_{L^q_tL^r_x}\to0,
\end{equation}
as $|\lambda|\to\infty$, where the $L^q_tL^r_x-$norms are taken in the space-time slab $[0, T]\times\R^N$. In particular, convergence holds in $\cfun([0, T];\Sigma(\R^N))$.
\end{itemize}
\end{theorem}
Let us recall the definition of $\Phi^\lambda$ given in Lemma \ref{lemma:transf}, then \eqref{eq:conv} yields
\begin{equation*}
\Psi^\lambda(t)-\Omega_\lambda(t)VU(t)\to0,\quad\textrm{in}\;L^q([0, T]; L^r(\R^N)),
\end{equation*}
as $|\lambda|\to\infty$, and the same holds for its gradient and for the multiplication by $|x|$. This is exactly what is stated in formula \eqref{eq:psi_conv}.
Hence Theorem \ref{thm:main1} directly follows from the one above.
\newline
The idea of the proof can be described as follows; let us consider the equation satisfied by the difference $\Phi^\lambda-U$. By using the integral formulation we may write
\begin{equation*}
\begin{aligned}
\Phi^\lambda(t)-U(t)=&-i\int_0^tS_0(t-s)\left[\hat B^\infty[\Phi^\lambda]\Phi^\lambda-\hat B^\infty[U]U\right](s)ds\\
&-i\int_0^tS_0(t-s)R^\lambda[\Phi^\lambda]\Phi^\lambda(s)ds=:-iI_1-iI_2,
\end{aligned}
\end{equation*}
see \eqref{eq:b_infty}, \eqref{eq:R_lambda} for the definitions of $\hat B^\infty, R^\infty$.
\newline
The oscillating coefficients in $R^\lambda$ converge weakly to zero, and since they appear inside the time integral in the Duhamel's formula, then $I_2$ converges (strongly) to zero. For the $I_1$ part, on the other hand, we can use a Lipschitz estimate for the nonlinearity $\hat B^\infty[F]F$ and close the convergence with a continuity argument.
\newline
To prove Theorem \ref{thm:main} above we proceed in two steps: first we prove that, as long as we have uniform bounds on $\Phi^\lambda$ on a space-time slab, we obtain the convergence. Then we prove that we indeed have such bounds on $\Phi^\lambda$ in
$[0, T]\times\R^N$, with $0<T<S_{max}$, provided $|\lambda|$ is sufficiently large.
\newline
We start by giving a technical Lemma which will be used to prove the convergence of $I_2$ to zero.
\begin{lemma}\label{lemma:weak_conv}
Let $(\tilde q, \tilde r)$ be an admissible pair, $0<T<\infty$, and let 
$f\in L^{\tilde q'}([0, T]; L^{\tilde r'}(\R^N))$. Then, for any admissible pair $(q, r)$, we have
\begin{equation}\label{eq:93}
\|\int_0^tS_0(t-s)e^{ia\lambda s}f(s)ds\|_{L^q([0, T];L^r(\R^N))}\to0,\quad\textrm{as}\;|\lambda|\to\infty,
\end{equation}
where $a$ can be any non-zero real number.
\end{lemma}
\begin{proof}
First of all, let us notice that by Strichartz estimates on the space-time slab $\displaystyle[0, T]\times\R^N$ we have
\begin{equation*}
\|\int_0^tS_0(t-s)e^{ia\lambda s}f(s)ds\|_{L^q_tL^r_x}\lesssim\|f\|_{L^{\tilde q'}_tL^{\tilde r'}_x}.
\end{equation*}
Then, by using a standard density argument, we see it suffices to prove \eqref{eq:93} for all $f\in\cfun^1([0, T];\mathcal S(\R^N))$. Let us consider the integral in \eqref{eq:93}, by integration by parts we obtain
\begin{equation*}
\int_0^te^{ia\lambda s}S_0(t-s)f(s)ds=\frac{1}{ia\lambda}\pare{e^{ia\lambda t}f(t)-S_0(t)f(0)}
-\frac{1}{ia\lambda}\int_0^tS_0(t-s)\left[\d_sf(s)-Hf(s)\right]ds.
\end{equation*}
Thus again by Strichartz estimates we get
\begin{equation*}
\|\int_0^tS_0(t-s)e^{ia\lambda s}f(s)ds\|_{L^q_tL^r_x}\lesssim
\frac{1}{|\lambda|}\pare{\|f\|_{L^q_tL^r_x}+\|f(0)\|_{L^2}+\|\d_sf-Hf\|_{L^{\tilde q'}_tL^{\tilde r'}_x}},
\end{equation*}
and the right hand side goes to zero as $|\lambda|\to\infty$.
\end{proof}
The next Proposition shows that, as long as we have uniform bounds on the $L^\infty_t\Sigma_x-$norm of $\Phi^\lambda$ on a space-time slab, then we can prove the convergence of $\Phi^\lambda$ towards $U$ in that space-time slab.
\begin{proposition}\label{lemma:14} 
Let $\Phi_0\in \Sigma(\R^N)$ and let $\Phi^{\lambda}$ denote the maximal solution of \eqref{eq:phi_sys}. Let $U$ be the maximal solution of \eqref{eq:u_sys}, defined on $[0,S_{max})$. Let $0<\ell<S_{max}$ and assume that $\Phi^{\lambda}$ exists on $[0,\ell]\times\R^N$ and that
\begin{equation}\label{eq:94}
\limsup_{|\lambda| \rightarrow \infty} \|\Phi^{\lambda}\|_{L^{\infty}([0,\ell];\Sigma(\R^N))}< \infty.
\end{equation}
Then we have
\begin{equation}
\lim_{|\lambda| \rightarrow \infty}  \| \left(\begin{array}{c} 1\\\na \\|\cdot| \end{array}\right)(\Phi^{\lambda}-U)\|_{L^q([0,\ell];L^r(\R^N))}=0
\end{equation}
for any admissible pairs $(q,r)$. In particular $\Phi^{\lambda}\rightarrow U$ in $L^{\infty}([0,\ell];\Sigma(\R^N))$.
\end{proposition}
\begin{proof}
By hypothesis \eqref{eq:94}, we can take $L$ large enough such that
\begin{equation*}
\sup_{|\lambda|\geq L}\|\Phi^{\lambda}\|_{L^{\infty}([0,\ell];\Sigma(\R^N))} < \infty.
\end{equation*}
By using Strichartz estimates similarly to the proof of Proposition \ref{prop:LWP} we see that from the bound above we also infer for a space-time slab $[0,\ell]\times\R^N$:
\begin{equation}\label{eq:95}
\sup_{|\lambda|\geq L}\sup_{(q, r)}
\|\left(\begin{array}{c}1\\\nabla\\|\cdot|\end{array}\right)\Phi^\lambda\|_{L^q_tL^r_x}<\infty,
\end{equation}
where the second supremum is taken over all admissible pairs $(q, r)$. If $\Phi^\lambda$ is a solution to \eqref{eq:phi_sys} and $U$ to \eqref{eq:u_sys}, then we have
\begin{equation*}
\begin{aligned}
\Phi^\lambda(t)-U(t)=&-i\int_0^tS_0(t-s)\left(\hat B^\infty[\Phi^\lambda]\Phi^\lambda-\hat B^\infty[U]U\right)(s)ds\\
&-i\int_0^tS_0(t-s)R^\lambda[\Phi^\lambda]\Phi^\lambda(s)ds=:-iI_1-iI_2,
\end{aligned}
\end{equation*}
where $\hat B^\infty, R^\lambda$ are defined in \eqref{eq:b_infty}, \eqref{eq:R_lambda}. We first consider $I_1$; by using Strichartz estimates, inequality \eqref{eq:11}, H\"older's inequality and then \eqref{eq:95}, we have 
\begin{equation*}
\begin{aligned}
\|I_1\|_{L^q_tL^r_x}\leq &C (\|\Phi^{\lambda}\|^2_{L^{\infty}_tL^4_x}+\|U\|^2_{L^{\infty}_tL^4_x}) 
\| \Phi^{\lambda}-U\|_{L^{\frac{8}{8-N}}_tL^4_x}\\
\leq&  C  \| \Phi^{\lambda}-U\|_{L^{\frac{8}{8-N}}_tL^4_x}.
\end{aligned}
\end{equation*}
For $I_2$ we use Lemma \ref{lemma:weak_conv}, indeed by Sobolev embedding we have
\begin{equation*}
\|R^\lambda[\Phi^\lambda]\Phi^\lambda\|_{L^{\frac{8}{8-N}}_tL^{4/3}_x}
\lesssim\|\Phi^\lambda\|_{L^\infty_tL^4_x}^3
\lesssim\|\Phi^\lambda\|_{L^\infty_tH^1_x}^3,
\end{equation*}
and consequently
\begin{equation*}
\|\int_0^tS_0(t-s)R^\lambda[\Phi^\lambda]\Phi^\lambda(s)ds\|_{L^q_tL^r_x}
\end{equation*}
converges to zero as $|\lambda|\to0$, for any admissible pair $(q, r)$, because every cubic term in 
$R^\lambda[\Phi^\lambda]\Phi^\lambda$ has an oscillatory coefficient in front of it.
By choosing $(q, r)=(\frac{8}{N}, 4)$, we have
\begin{equation*}
\|\Phi^\lambda-U\|_{L^{8/N}_tL^4_x}\leq\eps+
C\|\Phi^\lambda-U\|_{L^{\frac{8}{8-N}}_tL^{4}_x}.
\end{equation*}
A standard continuity argument gives us
\begin{equation*}
\|\Phi^\lambda-U\|_{L^{8/N}_tL^4_x}\leq C\eps,
\end{equation*}
which proves the convergence $\Phi^\lambda\to U$ in $L^{8/N}_tL^4_x$. 
The same holds for any admissible pair $(q, r)$, since
\begin{equation*}
\|\Phi^\lambda-U\|_{L^q_tL^r_x}\leq\|I_1\|_{L^q_tL^r_x}+\|I_2\|_{L^q_tL^r_x}
\leq\eps+C\|\Phi^\lambda-U\|_{L^{\frac{8}{8-N}}_tL^4_x}
\leq\eps+C\|\Phi^\lambda-U\|_{L^{8/N}_tL^4_x},
\end{equation*}
where in the last inequality we used H\"older in time.
Consequently $\Phi^\lambda\to U$ in $L^q_tL^r_x([0, \ell]\times\R^N)$ for any admissible pair. We now prove the convergence of $\nabla\Phi^\lambda$ and $|\cdot|\Phi^\lambda$ in the same spaces. Again we consider the difference
\begin{equation*}
\nabla(\Phi^\lambda-U)+x(\Phi^\lambda-U)=:-iI_1-iI_2-iI_3,
\end{equation*}
where, by using the commutator relations of the Hamiltonian $H$ with $x$ and $\nabla$, see \eqref{eq:comm}, we obtain
\begin{equation*}
\begin{aligned}
I_1:=&\int_0^tS_0(t-s)\nabla\left(\hat B^\infty[\Phi^\lambda]\Phi^\lambda-\hat B^\infty[U]U\right)(s)ds,\\
I_2:=&\int_0^tS_0(t-s)x\left(\hat B^\infty[\Phi^\lambda]\Phi^\lambda-\hat B^\infty[U]U\right)(s)ds,\\
I_3:=&\int_0^tS_0(t-s)\left(\nabla(R^\lambda[\Phi^\lambda]\Phi^\lambda)+xR^\lambda[\Phi^\lambda]\Phi^\lambda\right)(s)ds.
\end{aligned}
\end{equation*}
By using Strichartz estimates, inequality \eqref{eq:12} and Sobolev embedding, we have
\begin{equation*}
\|I_1\|_{L^{8/N}_tL^4_x}\leq C\left(\|\Phi^\lambda\|_{L^\infty_tH^1_x}^2+\|U\|_{L^\infty_tH^1_x}^2\right)
\|\nabla(\Phi^\lambda-U)\|_{L^{\frac{8}{8-N}}_tL^{4}_x}.
\end{equation*}
Analogously, for $I_2$ we have
\begin{equation*}
\|I_2\|_{L^{8/N}_tL^4_x}\leq C\left(\|\Phi^\lambda\|_{L^\infty_tH^1_x}^2+\|U\|_{L^\infty_tH^1_x}^2\right)
\||\cdot|(\Phi^\lambda-U)\|_{L^{\frac{8}{8-N}}_tL^{4}_x}.
\end{equation*}
Finally, for $I_3$ we use inequality \eqref{eq:13}, Sobolev embedding and \eqref{eq:95} to estimate
\begin{equation*}
\|\nabla(R^\lambda[\Phi^\lambda]\Phi^\lambda)+|\cdot|R^\lambda[\Phi^\lambda]\Phi^\lambda\|_{L^{\frac{8}{8-N}}_tL^{4/3}_x}
\leq C\|\Phi^\lambda\|_{L^\infty_tH^1_x}^2\left(\|\nabla\Phi^\lambda\|_{L^{8/N}_tL^4_x}
+\||\cdot|\Phi^\lambda\|_{L^{8/N}_tL^4_t}\right)<\infty.
\end{equation*}
Hence we can apply Lemma \ref{lemma:weak_conv} and infer that
\begin{equation*}
\|I_3\|_{L^q_tL^r_x}\to0,\textrm{as}\;|\lambda|\to\infty,
\end{equation*}
for any $(q, r)$ admissible pair. By resuming, we have obtained
\begin{equation*}
\|\nabla(\Phi^\lambda-U)\|_{L^{8/N}_tL^4_x}
+\||\cdot|(\Phi^\lambda-U)\|_{L^{8/N}_tL^4_x}\leq\eps
+C\left(\|\nabla(\Phi^\lambda-U)\|_{L^{8/N}_tL^4_x}
+\||\cdot|(\Phi^\lambda-U)\|_{L^{8/N}_tL^4_x}\right).
\end{equation*}
As before, a standard continuity argument gives us 
\begin{equation*}
\|\nabla(\Phi^\lambda-U)\|_{L^{8/N}_tL^4_x}
+\||\cdot|(\Phi^\lambda-U)\|_{L^{8/N}_tL^4_x}\leq C\eps,
\end{equation*}
which shows the convergence in $L^{8/N}_tL^4_x([0, \ell]\times\R^N)$ as $|\lambda|\to\infty$. The convergence in $L^q_tL^r_x([0, \ell]\times\R^N)$ for any $(q, r)$ admissible pair then follows as before, by using Strichartz estimates.
\end{proof}
The Lemma above states that, as long as the $L^{\infty}([0, \ell]; \Sigma(\R^N))-$norm of the family $\{\Phi^\lambda\}$ of solutions for the Cauchy problem related to \eqref{eq:phi_sys} stays eventually bounded (as $|\lambda|\to\infty$) in a time interval $[0, \ell]$, then the family converges in to $U$ solution to \eqref{eq:u_sys} with the same initial datum, in every Strichartz space.
\newline
It thus remains to prove that given any time interval $[0, T]$ with $0<T<S_{max}$, the $L^{\infty}([0, T]; \Sigma(\R^N))-$norm of 
$\Phi^\lambda$ is indeed uniformly bounded, if we take $|\lambda|$ sufficiently large.
\newline
\emph{Proof of Theorem \ref{thm:main}.}
Let then $T$ be a positive time strictly less than the maximal existence time for the solution $U$ of 
\eqref{eq:u_sys}, i.e. $0<T<S_{max}$, and let us fix the constant
\begin{equation*}
M:=2C\|U\|_{L^\infty([0, T]; \Sigma(\R^N))},
\end{equation*}
where $C$ is the constant in \eqref{eq:lin_pert}.
Let $\delta=\delta(M)$ be the constant in Proposition \ref{prop:LWP}. Then $\Phi^\lambda$ exists in $[0, \delta]$ for all $\lambda$ and furthermore
\begin{equation*}
\sup_{\lambda\in\R}\|\Phi^\lambda\|_{L^\infty((0, \delta); \Sigma(\R^N))}\leq 2C\|\Phi_0\|_{\Sigma(\R^N)},
\end{equation*}
by \eqref{eq:lin_pert}.
Now, let $0<\ell\leq T$ be such that $\Phi^\lambda$ exists in $[0, \ell]$, and that we have
\begin{equation*}
\limsup_{|\lambda|\to\infty}\|\Phi^\lambda\|_{L^\infty((0, \ell); \Sigma(\R^N))}<\infty.
\end{equation*}
Notice that by the inequality above we see that we can always choose $\ell=\delta$. Thus, by Proposition \ref{lemma:14} we have
\begin{equation*}
\lim_{|\lambda|\to\infty}\left(
\|\Phi^\lambda-U\|_{L^q((0, \ell); L^r(\R^N))}
+\|\nabla(\Phi^\lambda-U)\|_{L^q((0, \ell); L^r(\R^N))}
+\||\cdot|(\Phi^\lambda-U)\|_{L^q((0, \ell); L^r(\R^N))}
\right)=0
\end{equation*}
for all admissible pairs $(q, r)$. In particular we have
\begin{equation*}
\lim_{|\lambda|\to\infty}\|\Phi^\lambda(\ell)-U(\ell)\|_{\Sigma(\R^N)}=0.
\end{equation*}
This implies that 
\begin{equation*}
\sup_{|\lambda|\geq\Lambda}\|\Phi^\lambda(\ell)\|_{\Sigma(\R^N)}\leq M,
\end{equation*}
for some $\Lambda>0$ sufficiently large. \newline
We can thus repeat the same argument, starting at time $t=\ell$. Consequently we have the solution $\Phi^\lambda$ exists in 
$[0, \delta+\ell]$ and we have a uniform bound for the $\Sigma-$norm of $\Phi^\lambda$,
\begin{equation*}
\sup_{|\lambda|\geq\Lambda}\|\Phi^\lambda\|_{L^\infty((0, \delta+\ell); \Sigma(\R^N))}\leq 2CM.
\end{equation*}
Thus we can apply again Lemma \ref{lemma:14} in the time interval $[\ell, \delta+\ell]$ and obtain the convergence of 
$\Phi^\lambda$ to $U$. Again, because of the convergence we have
\begin{equation*}
\lim_{|\lambda|\to\infty}\|\Phi^\lambda(\delta+\ell)-U(\delta+\ell)\|_{\Sigma(\R^N)}=0,
\end{equation*}
and in particular
\begin{equation*}
\sup_{|\lambda|\geq\Lambda'}\|\Phi^\lambda(\delta+\ell)\|_{\Sigma(\R^N)}\leq M.
\end{equation*}
Thus we repeat this argument to prove the result in the whole time interval $[0, T]$ and Theorem \ref{thm:main} is proved.
\qed\newline
\section*{Acknowledgements}
 The work of the second author has been supported by the Hertha-Firnberg Program 
 of the Austrian Science Fund (FWF), grant T402-N13.
\appendix 
\section{Simple example}
To illustrate the application of the presented theory we give a very simple example.
Let us consider the case of one focusing and one defocusing nonlinearity, where the interspecific scattering length is zero. Thus we consider the following system in $N=2$:
\begin{equation}\label{eq:example}
i\d_t\Psi=-\frac12\Delta\Psi+\frac{\gamma^2}{2}|x|^2\Psi+\tilde B[\Psi]\Psi
+A\Psi,
\end{equation}
with $A$ as defined in \eqref{eq:A}, and
\begin{align*}
\tilde B[\Psi]=&\left(\begin{array}{cc}-\beta |\psi_1|^2&0\\
0&\beta|\psi_2|^2\end{array}\right)\quad \beta>0.
\end{align*}
We choose initial data such that the condition (3) in Theorem \ref{thm:GWP1} is not satisfied, thus the initial mass M(0) is not too small, and none of the conditions (i)-(iii) of Theorem \ref{thm:blow-up} is satisfied. In this way, we are actually not able to say if the solution may blow-up or exist globally. 
\begin{remark}
If $\lambda=0$ the system is decoupled and we have two scalar nonlinear Schr\"odinger equations, for which we know that $\psi_2$ exists globally, and $\psi_1$ may blow up in finite time for initial mass $\| \psi_{1, 0}\|_{L^2(\R^2)}\geq \|Q\|_{L^2(\R^N)}$ with $Q$ being the unique positive radial solution to $\Delta Q+Q^3-Q=0$.
\end{remark}
Performing the asymptotics for $|\lambda| \rightarrow \infty$ we observe that $\Psi$ converges to $U$ solution to the linear Schr\"odinger system:
\begin{equation*}
i\d_t U =-\frac{1}{2}\Delta U+\frac{\gamma^2}{2}|x|^2 U.
\end{equation*}
Applying Theorem \ref{thm:main}  we can say that the solution $\Psi$ of \eqref{eq:example} exists in $[0,T]$, for any $T< \infty$, provided that $\lambda$ is taken sufficiently large. 
\bibliography{mybibasymptotics}{}

\begin{thebibliography}{10}

\bibitem{Abdullaev}
F.K. Abdullaev, J.G. Caputo, R.~A. Kraenkel, and B.A. Malomed.
\newblock Controlling collapse in {B}ose-{E}instein condensates by temporal
  modulation of the scattering length.
\newblock {\em Phys. Rev. A}, 67(013605), 2003.

\bibitem{AS10}
P.~Antonelli and C.~Sparber.
\newblock Global well-posedness for cubic {NLS} with nonlinear damping.
\newblock {\em Comm. Partial Differential Equations}, 35:122310--2328, 2010.

\bibitem{Ballagh}
R.~J. Ballagh, K.~Burnett, and T.~F. Scott.
\newblock Theory of an output coupler for {B}ose-{E}instein condensed atoms.
\newblock {\em Phys. Rev. Lett.}, 78:1607--1611, 1997.

\bibitem{Bao}
W.~Bao and Y.~Cai.
\newblock Ground states of two-component {B}ose-{E}instein condensates with an
  internal josephson junction.
\newblock {\em East Asian J. Appl. Math.}, 1(1):49--81, 2011.

\bibitem{Car}
R.~Carles.
\newblock Remarks on nonlinear {S}chr{\"o}dinger equations with harmonic
  potential.
\newblock {\em Ann. Henri Poincare}, 3:757--772, 2002.

\bibitem{Car05}
R.~Carles.
\newblock Global existence results for nonlinear {S}chr{\"o}dinger equations
  with quadratic potentials.
\newblock {\em Discrete Cont. Dyn. Syst.}, 13(2):385--398, 2005.

\bibitem{CarGamPan}
X.~Carvajal, P.~Gamboa, and M.~Panthee.
\newblock A system of coupled {S}chr{\"o}dinger equations with time-oscillating
  nonlinearity.
\newblock {\em Int. J. Math.}, 23(11):1250119, 2012.

\bibitem{MR2809621}
X.~Carvajal, M.~Panthee, and M.~Scialom.
\newblock On the critical {KDV} equation with time-oscillating nonlinearity.
\newblock {\em Differential Integral Equations}, 24(5-6):541--567, 2011.

\bibitem{Caz10}
T.~Cazenave and M.~Scialom.
\newblock A {S}chr{\"o}dinger equation with time-oscillationg nonlinearity.
\newblock {\em Revista Mat\'ematica Complutense}, 23(2):321--339, 2010.

\bibitem{ChenWei}
G.~Chen and Y.~Wei.
\newblock Energy criteria of global existence for the coupled nonlinear
  {S}chr{\"o}dinger equations with harmonic potentials.
\newblock {\em NoDEA Nonlinear Equations Appl.}, 15:195--208, 2008.

\bibitem{Decon}
B.~Deconinck, P.G. Kevrekidis, H.E. Nistazakis, and D.J. Frantzeskakis.
\newblock Linearly coupled {B}ose-{E}instein condensates: From {R}abi
  oscillations and quasi-periodic solutions to oscillating domain walls and
  spiral waves.
\newblock {\em Phys. Rev. A}, 70(063605), 2004.

\bibitem{FanMont}
L.~Fanelli and E.~Montefusco.
\newblock On the blow-up threshold for weakly coupled nonlinear
  {S}chr{\"o}dinger equations.
\newblock {\em J. Phys. A: Math. Theor.}, 40:14139--14150, 2007.

\bibitem{Fang}
D.~Fang and Z.~Han.
\newblock A {S}chr{\"o}dinger equation with time-oscillating critical
  nonlinearity.
\newblock {\em Nonliear Analysis}, 74:4698--4708, 2011.

\bibitem{G}
R.~T. Glassey.
\newblock On the blowing up of solutions to the {C}auchy problem for nonlinear
  {S}chr{\"o}dinger equations.
\newblock {\em J. Math. Phys.}, 18(9):1794--1797, 1977.

\bibitem{Hall:1998fk}
D.S. Hall, M.R. Matthews, C.E. Wieman, and E.A. Cornell.
\newblock Measurements of relative phase in two-component bose-einstein
  condensates.
\newblock {\em Physical Review Letters}, 81(8):1543--1546, 1998.

\bibitem{Juengel11}
A.~J{\"u}ngel and R.~Weish{\"a}upl.
\newblock Blow-up in two-component nonlinear {S}chr{\"o}dinger systems with an
  external driven field.
\newblock {\em Math. Models Methods Appl. Sci.}, Available online
  doi:10.1142/S0218202513500206(in press), 2013.

\bibitem{Kato87}
T.~Kato.
\newblock On nonlinear {S}chr{\"o}dinger equations.
\newblock {\em Ann. Inst. Henri Poincar{\'e}}, 46(1):113--129, 1987.

\bibitem{LiWuLai}
X.~Li, Y.~Wu, and S.~Lai.
\newblock A sharp threshold of blow-up for coupled nonlinear {S}chr{\"o}dinger
  equations.
\newblock {\em J. Phys. A: Math. Theor.}, 43:165205--165216, 2010.

\bibitem{LinWei}
T.-C. Lin and J.~Wei.
\newblock Solitary and self-similar solutions of two-component system of
  nonlinear {S}chr{\"o}dinger equations.
\newblock {\em Physica D}, 220:99--115, 2006.

\bibitem{MR2449345}
Zuhan Liu.
\newblock Two-component {B}ose-{E}instein condensates.
\newblock {\em J. Math. Anal. Appl.}, 348(1):274--285, 2008.

\bibitem{MaZhao}
L.~Ma and L.~Zhao.
\newblock Sharp thresholds of blow-up and global existence for the coupled
  nonlinear {S}chr{\"o}dinger system.
\newblock {\em J. Math. Phys.}, 49:062103--061220, 2008.

\bibitem{matthews1999vortices}
M.R. Matthews, B.P. Anderson, P.C. Haljan, D.S. Hall, C.E. Wieman, and E.A.
  Cornell.
\newblock Vortices in a bose-einstein condensate.
\newblock {\em Physical Review Letters}, 83(13):2498--2501, 1999.

\bibitem{Myatt97}
C.~J. Myatt, E.~A. Burt, R.~W. Ghrist, E.~A. Cornell, and C.~E. Wieman.
\newblock Production of two overlapping {B}ose-{E}instein condensates by
  sympathetic cooling.
\newblock {\em Phys. Rev. Lett.}, 78:586--589, 1997.

\bibitem{Nista}
H.E. Nistazakis, Z.~Rapti, D.J. Frantzeskakis, P.G. Kevrekidis, P.~Sodano, and
  A.~Trombettoni.
\newblock Rabi switch of condensate wavefunctions in a multicomponent bose gas.
\newblock {\em Phys. Rev. A}, 78(023635), 2008.

\bibitem{Oh}
Y.G. Oh.
\newblock Cauchy problem and {E}hrenfest's law of nonlinear {S}chr{\"o}dinger
  equations with potentials.
\newblock {\em J. Diff. Eq.}, 81:255--274, 1989.

\bibitem{Prytula}
V.~Prytula, V.~Vekslerchik, and V.~P{\'e}rez-Garc\'{\i}a.
\newblock Collapse in coupled nonlinear {S}chr{\"o}dinger equations: Sufficient
  conditions and applications.
\newblock {\em Physica D}, 238:1462--1467, 2009.

\bibitem{Saito}
H.~Saito, R.G.Hulet, and M.Ueda.
\newblock Stabilization of a {B}ose-{E}instein droplet by hyperfine {R}abi
  oscillations.
\newblock {\em Phys. Rev. A}, 76(053619), 2007.

\bibitem{MR2283958}
B.~Sirakov.
\newblock Least energy solitary waves for a system of nonlinear {S}chr\"odinger
  equations in {$\Bbb R^n$}.
\newblock {\em Comm. Math. Phys.}, 271(1):199--221, 2007.

\bibitem{Caz}
T.Cazenave.
\newblock {\em Semilinear Schr{\"o}dinger Equations}, volume~10.
\newblock American Mathematical Society, 2003.

\bibitem{Weinstein}
M.~Weinstein.
\newblock Nonlinear {S}chr{\"o}dinger equations and sharp interpolation
  estimates.
\newblock {\em Commun. Math. Phys.}, 87, 1983.

\bibitem{Will99}
J.~Williams, R.~Walser, J.~Cooper, E.~Cornell, and M.~Holland.
\newblock Nonlinear {J}osephson-type oscillations of a driven, two-component
  {B}ose-{E}instein condensate.
\newblock {\em Phys.Rev.A}, 59:R31--R34, 1999.

\bibitem{williams1999preparing}
J.E. Williams and M.J. Holland.
\newblock Preparing topological states of a bose--einstein condensate.
\newblock {\em Nature}, 401(6753):568--572, 1999.

\end{thebibliography}
\bibliographystyle{plain}
\end{document}